\documentclass[12pt,reqno]{amsart}
\usepackage{amsmath,amsfonts,graphicx,amscd,amssymb,epsf,amsthm,enumerate,parskip,mathrsfs,color,ltablex,blindtext,cancel,cases}  
\usepackage[a4paper, margin=1in]{geometry}
\usepackage{comment}
\usepackage{mathtools}
\usepackage{times}
\usepackage{graphicx}
\usepackage{float}
\theoremstyle{definition}
\newtheorem{theorem}{Theorem}[section]
\newtheorem{proposition}[theorem]{Proposition}
\newtheorem{lemma}[theorem]{Lemma}
\newtheorem{corollary}[theorem]{Corollary}
\newtheorem{definition}[theorem]{Definition}
\newtheorem{example}[theorem]{Example}
\newtheorem{question}[theorem]{Question}

 \numberwithin{equation}{section}
\numberwithin{equation}{section}
 
\newcommand{\RR}{\mathbb{R}}
\newcommand{\FF}{\mathbb{F}}
\newcommand{\PP}{\mathbb{P}}
\newcommand{\CC}{\mathbb{C}}
\newcommand{\ZZ}{\mathbb{Z}}

\newcommand{\scal}{\text{scal}}

\newcommand{\sys}{\operatorname{sys}}
\newcommand{\bl}{\operatorname{Bl}}
\newcommand{\sysh}{\operatorname{sys}^{\operatorname{hol}}}
\newcommand{\sysp}{\operatorname{sys}^{\pi_2}_2}
\newcommand{\mori}{\overline{\operatorname{NE}}}
\DeclareMathOperator{\Vol}{Vol}
\DeclareMathOperator{\Area}{Area}

\DeclareMathOperator{\ric}{Ric}

\usepackage{cite}
\usepackage[colorlinks,citecolor=blue]{hyperref}
\usepackage{amssymb}

\begin{document}
\makeatletter
\newcommand{\rmnum}[1]{\romannumeral #1}
\newcommand{\Rmnum}[1]{\expandafter\@slowromancap\romannumeral #1@}
\makeatother
\title{The 2-systole on compact Kähler surfaces with positive scalar curvature}
 \author{Zehao Sha}
\begin{abstract}
We study the 2-systole on compact K\"ahler surfaces of positive scalar curvature. For any such surface $(X,\omega)$, we prove the sharp estimate $\min_X S(\omega)\cdot\operatorname{sys}_2(\omega)\le12\pi$, with equality if and only if $X=\mathbb P^2$ and $\omega$ is the Fubini--Study metric. Using the classification of positive scalar curvature K\"ahler surfaces, we determine the optimal constant in the case that $X$ is rational or non-rational and describe the corresponding rigid models. When $X$ is a non-rational ruled surface, we also give an independent analytic proof, adapting Stern's level set method to the holomorphic fibration in the K\"ahler setting.
\end{abstract}
\maketitle
 
\section{Introduction}

Systolic geometry studies quantitative lower bounds for nontrivial homology classes represented by cycles of minimal volume. Following Berger's terminology \cite{Berger72Pu,Berger72Loewner,BergerIsosystolic} (see also \cite{BergerWhatIsSystole}) and Gromov's subsequent development of systolic geometry \cite{Gromov83}, if $(M,g)$ is a closed Riemannian manifold of dimension $n\ge k$, the homological $k$--systole of $(M,g)$ is defined by
\[
\sys_k(g):= \inf\bigl\{\Vol_g(Z)\ \big|\ Z \text{ is a }k\text{-cycle with }[Z]\neq0\in H_k(M;\mathbb Z)\bigr\}.
\]
Here an integral $k$-cycle may be understood, for instance, as an integral current in the sense of geometric measure theory, and $\Vol_g(Z)$ denotes its $k$-dimensional mass with respect to $g$.

In the setting of \emph{positive scalar curvature} (hereafter PSC), a particularly natural object is the \emph{2-systole}. Indeed, scalar curvature is obtained by averaging sectional curvatures over $2$-planes, and a basic mechanism for producing PSC metrics is to insert an $S^2$-component: for product metrics, one has \(\scal(g\oplus h)=\scal(g)+\scal(h)\), so $S^2(r)\times (N,h)$ carries PSC metrics once $r$ is sufficiently small, to a large extent, independently of the geometry of $N$. This suggests that, in many PSC constructions, the relevant geometric scale is encoded by a $2$-sphere direction, and hence by a $2$-systolic quantity.

In dimension three, the interplay between $2$-systole and PSC is by now well understood. If $(M^3,g)$ is PSC, Schoen-Yau \cite{SchoenYau1979} proved that any area--minimizing surface in $M$ is homeomorphic to either $S^2$ or $\mathbb{RP}^2$, showing that area-minimizing surfaces in the PSC setting are forced to have spherical topology. Building on this, Bray-Brendle-Neves \cite{BBN2010} established a sharp $\pi_2$-systolic inequality. Denote by \(\sysp(g)\) the infimum of the areas of homotopically nontrivial $2$--spheres in $(M^3,g)$. Then
\begin{equation}\label{eq:BBN}
  \min_M \scal(g)\cdot \sysp(g)\;\le 8\pi,
\end{equation}
with equality if and only if the universal cover of $(M^3,g)$ is isometric to the Riemannian product $S^2\times\mathbb R$ endowed with the round metric on $S^2$ and the flat metric on $\mathbb R$. In \cite{stern2022scalar}, Stern gave a new proof of \eqref{eq:BBN} leveraging the level set method. The bound \eqref{eq:BBN} was recently refined by a quantitative gap theorem of Xu \cite{Xu2025}, and there has been substantial further progress on systolic inequalities under scalar curvature assumptions; see for instance \cite{ZhuPAMS2020,Richard2020}.

In real dimension four, by contrast, our current understanding of $2$-systoles under PSC assumptions is much more limited: even on $S^2\times S^2$ with positive scalar curvature, global upper bounds for $\sys_2(g)$ are only known under additional geometric hypotheses. In this paper we initiate a systematic study of $2$-systolic inequalities in dimension four under the extra assumption that the metric is K\"ahler. More precisely, given a compact K\"ahler surface $(X,\omega)$ we denote by $g_\omega$ the associated Riemannian metric and we write
\[
  \sys_2(\omega) := \sys_2(g_\omega).
\]
We also write $S(\omega)$ for the Chern scalar curvature of $\omega$; for K\"ahler metrics, one has $2S(\omega)=\scal(g_\omega)$), so positivity of $S(\omega)$ is equivalent to positive Riemannian scalar curvature. With this convention, all our main inequalities are stated purely in terms of the K\"ahler form
$\omega$. 
\begin{theorem}\label{thm:intro-global}
Let $(X,\omega)$ be a compact K\"ahler surface with $S(\omega)>0$. Then
\[
\min_X S(\omega)\cdot \sys_2(\omega)\;\le\;12\pi.
\]
Moreover, equality holds if and only if $X\cong \PP^2$ and $\omega$ is the Fubini--Study metric.
\end{theorem}

The birational classification of PSC Kähler surfaces is by now complete. Yau proved that a compact Kähler surface with positive total scalar curvature has Kodaira dimension $-\infty$ \cite{Yau74}. By the Enriques--Kodaira classification, a minimal such surface is therefore either $\PP^2$ or ruled \cite[Chapter~V]{BHPV04}. In addition, LeBrun showed that for minimal complex surfaces of Kähler type, the existence of a PSC Riemannian metric is equivalent to the existence of a PSC Kähler metric, and this occurs precisely for $\PP^2$ and ruled surfaces \cite[Theorem~4]{LeBrun95}. The remaining birational step was completed by Brown, who proved that blow-ups preserve the sign of scalar curvature for sufficiently small data \cite[Theorem~B]{Brown24}. Thus a compact Kähler surface admits a PSC Kähler metric if and only if it is birational to $\PP^2$ or to a ruled surface. It is therefore natural to restate Theorem~\ref{thm:intro-global} in a form that records how the optimal constant depends on the classification. This leads to the following refinement.

\begin{theorem}\label{thm:intro-refined}
Let $(X,\omega)$ be a compact K\"ahler surface with positive scalar curvature.
Then:
\begin{enumerate}[\normalfont(1)]
  \item \emph{(Theorem \ref{thm:hirzebruch})}
  If $X$ is a rational surface such that $X \not\cong\PP^2$, then
  \[
  \min_X S(\omega)\cdot \sys_2(\omega)\;\le\;8\pi,
  \]
  with equality if and only if $X\cong \PP^1\times\PP^1$, $\omega$ is the product Fubini--Study metric on $\PP^1\times\PP^1$ up to scaling, and $\sys_2(\omega)$ is realized by a $\PP^1$--fibre.

  \item \emph{(Theorems \ref{thm:ruled-g>=1} and \ref{thm:ruled-g>=1-levelset})}
  If $X$ is a non-rational ruled surface fibred over a base curve $B$ of genus $g(B)\ge1$, then
  \[
  \min_X S(\omega)\cdot \sys_2(\omega)\;\le\;4\pi,
  \]
  with equality if and only if $B$ is an elliptic curve, the universal cover of $X$ is biholomorphic to $\PP^1\times\CC$, and $\omega$ is induced by the product of the Fubini--Study metric on $\PP^1$ and a flat metric on $\CC$, with $\sys_2(\omega)$ is realized by a $\PP^1$--fibre.
\end{enumerate}
\end{theorem}
\subsection*{Idea of the proof}

The guiding idea of this paper is to exploit the link between positive scalar curvature K\"ahler metrics and uniruledness to obtain geometric competitors for the $2$-systole. On a compact K\"ahler manifold, the Chern--Weil formula shows that positive total scalar curvature forces the canonical bundle to be non-pseudoeffective. By Ou's criterion \cite[Theorem~1.1]{ou2025characterization} (see also Cao--P\u{a}un \cite{CP}), such a manifold is therefore \emph{uniruled} (that is, covered by rational curves).

These curves provide natural candidates for the homological $2$-systole. Indeed, complex curves are calibrated by the K\"ahler form, so every nonzero effective curve $C$ satisfies
\[
\Area_\omega(C)=\int_C\omega=[\omega]\cdot[C]>0.
\]
In particular, $[C]\neq0$ in homology, and its area depends only on the K\"ahler class. This suggests estimating the
$2$-systole through the areas of effective curves, leading to the following algebraic analogue.

\begin{definition}[Holomorphic $2$-systole] \label{def:2-systole-kahler}
Let $X$ be a compact K\"ahler manifold containing a nonzero effective curve, and let $\alpha\in\mathcal K(X)$. The \emph{holomorphic $2$-systole} of $\alpha$ is defined by
\begin{equation}\label{eq:2-systole-kahler}
\sysh_2(\alpha):=\inf\bigl\{\alpha\cdot[C]\ \big|\ C\subset X \text{ an effective curve},\ 0\neq[C]\in H_2(X;\ZZ)\bigr\}.
\end{equation}
\end{definition}

The average Chern scalar curvature also depends only on the K\"ahler class. We therefore consider the scale-invariant functional
\[
\mathcal J_X(\alpha):=\sysh_2(\alpha)\cdot\hat S(\alpha).
\]
For a surface, writing $\alpha=[\omega]$, we have
\[
\mathcal J_X(\alpha)=4\pi\,\sysh_2(\alpha)\frac{c_1(X)\cdot\alpha}{\alpha^2}.
\]
The inequalities $\sys_2(\omega)\le\sysh_2(\alpha)$ and $\min_X S(\omega)\le\hat S(\alpha)$ imply that, for a PSC
K\"ahler metric,
\begin{equation} \label{eq:elementray-comparison}
    \min_X S(\omega)\cdot\sys_2(\omega)\le\mathcal J_X(\alpha).
\end{equation}
Thus our systolic estimates follow from upper bounds for $\mathcal J_X$. Moreover, equality in this comparison forces $\omega$
to be cscK and $\sys_2(\omega)=\sysh_2(\alpha)$.

The global estimate (Theorem \ref{thm:intro-global}) follows from Mori's bound on the length of extremal rays \cite{Mori82}. The key observation is that this bound translates the holomorphic $2$-systole into a nefness statement, reducing the desired inequality to positivity of intersection numbers. The same argument gives the estimate in arbitrary dimension for real ample classes on rationally connected manifolds.

For ruled surfaces, any nef class can be written in terms of the numerical classes of fibres, sections and exceptional divisors. We treat the rational and non-rational cases separately, reducing both to elementary intersection calculations. These give the constants $8\pi$ and $4\pi$, respectively, and apply to arbitrary configurations of blow-up centres. Tracking equality in these calculations, together with the elementrary comparison \eqref{eq:elementray-comparison}, identifies the equality models.

For non-rational ruled surfaces, we also give an independent analytic proof, inspired by Stern's level-set method \cite{stern2022scalar}. In the K\"ahler setting, the holomorphic fibration provides a natural replacement for the harmonic map used in the Riemannian argument. This yields the same sharp constant and the same rigidity statement as the algebro-geometric approach.

\subsection*{Why the K\"ahler setting is special}

In the general Riemannian setting, a uniform positive lower bound for scalar curvature does not control the homological $2$-systole in dimensions at least five. For example, on $M=S^2\times S^{n-2}$ with $n\ge5$, consider
\[
g_t=t^2g_{S^2}\oplus g_{S^{n-2}},\qquad t\ge1,
\]
where each $g_{S^m}$ is the round metric of sectional curvature $1$. Then
\[
\scal(g_t)=\frac{2}{t^2}+(n-2)(n-3),\qquad\sys_2(g_t)=4\pi t^2.
\]
Thus scalar curvature remains uniformly positive while the $2$-systole tends to infinity.

By contrast, K\"ahler geometry provides additional structure. In basic K\"ahler examples, such as projective space, the relevant even-dimensional homology groups are already nontrivial, so the corresponding systolic questions are non-vacuous. Our main approach uses the additional structure of K\"ahler geometry to relate scalar curvature to the areas of holomorphic curves through intersection theory.

\subsection*{Organization of the paper}

In Section~\ref{sec:preliminaries} we review the basic notions used throughout the paper.
In Section~\ref{sec:proof1}, we establish an algebraic estimate in arbitrary dimension using Mori's length bound and prove Theorem~\ref{thm:intro-global}.
In Section~\ref{sec:ruled}, we study ruled surfaces from an algebro-geometric perspective. We prove the refined inequalities by using intersection theory.
Finally, Section~\ref{sec:level-set} develops the level-set argument for holomorphic fibrations and gives an independent analytic proof of the non-rational inequality together with the rigidity statement.

\subsection*{AI Disclosure.}~~The idea and setup of this paper were developed by the author in early 2025. No AI was used in this paper.

\subsection*{Acknowledgment.}~~The author is deeply grateful to his advisor, Gérard Besson, for introducing him to the study of positive scalar curvature and for his constant encouragement. The author also thanks Garrett Brown and Jian Wang for valuable comments on an earlier version of the manuscript.
\section{Preliminaries}\label{sec:preliminaries}

Let \((X^n,\omega)\) be a compact K\"ahler manifold of complex dimension \(n\). We denote by
\[
H^{1,1}(X;\RR):=H^{1,1}(X)\cap H^2(X;\RR)
\]
the space of real \((1,1)\)-classes, and by
\[
\mathcal{K}(X):=\{[\omega]\in H^{1,1}(X;\RR)\mid [\omega]>0\}
\]
the K\"ahler cone of \(X\). It is an open convex cone in the finite-dimensional real vector space \(H^{1,1}(X;\RR)\).

The (Chern--)Ricci form is defined by
\[
  \ric(\omega):=-\sqrt{-1}\,\partial\bar{\partial}\log\det(\omega).
\]
It is a real closed \((1,1)\)-form representing the first Chern class:
\[
  c_1(X)=\frac{1}{2\pi}[\ric(\omega)]\in H^{1,1}(X;\RR).
\]
The (Chern--)scalar curvature \(S(\omega)\) is then determined by
\[
  S(\omega):=\frac{n\,\ric(\omega)\wedge\omega^{n-1}}{\omega^n}.
\]
In particular, the \emph{average scalar curvature}
\[
  \hat{S}([\omega]):= \frac{2n\pi\,c_1(X)\cup[\omega]^{n-1}}{[\omega]^n}
\]
depends only on the K\"ahler class \([\omega]\). Clearly,
\[
\min_X S(\omega)\le \hat{S}([\omega]),
\]
with equality if and only if \(\omega\) is cscK.

Throughout the paper we write \(g_\omega\) for the Riemannian metric associated to \(\omega\), and denote by
\[
  \sys_2(\omega):=\sys_2(g_\omega)
\]
the \(2\)-systole. In contrast, the quantity that we will actually compute is the following class-dependent holomorphic \(2\)-systole:
\[
\sysh_2([\omega]):= \inf\bigl\{[\omega]\cdot C \ \big|\ C\subset X \text{ an effective curve},\ [C]\neq 0 \text{ in } H_2(X;\ZZ)\bigr\}.
\]
Here an effective curve \(C=\sum_i m_i C_i\) is viewed as an integral real \(2\)--cycle, and we set
\[
  [\omega]\cdot C:=\sum_i m_i\int_{C_i^{\mathrm{reg}}}\omega.
\]
By abuse of notation, we do not distinguish a curve from its homology class when writing such pairings.

In particular, every effective curve determines a nontrivial integral \(2\)--cycle, and therefore contributes a competitor in the definition of the \(2\)-systole. Hence we always have
\begin{equation}\label{eq:syst-vs-sysh-notations}
  \sys_2(\omega)\le \sysh_2([\omega]).
\end{equation}
Note that \(\sysh_2([\omega])\) depends only on the K\"ahler class \([\omega]\), whereas \(\sys_2(\omega)\) a priori depends on the specific metric in that class.

The basic scale-invariant functional that we will work with is the function
\[
\mathcal{J}_X:\mathcal{K}(X)\to\RR,\qquad\mathcal{J}_X([\omega]):= \sysh_2([\omega])\cdot\hat{S}([\omega]).
\]
Equivalently,
\[
\mathcal{J}_X([\omega])= \sysh_2([\omega])\cdot\frac{2n\pi\,c_1(X)\cup[\omega]^{n-1}}{[\omega]^n}.
\]
For K\"ahler surfaces (\(n=2\)) this becomes
\[
\mathcal{J}_X([\omega])= \sysh_2([\omega])\cdot\frac{4\pi\,c_1(X)\cdot[\omega]}{[\omega]^2}.
\]
Since \(\sysh_2(\lambda[\omega])=\lambda\,\sysh_2([\omega])\) and \(\hat{S}(\lambda[\omega])=\lambda^{-1}\hat{S}([\omega])\) for \(\lambda>0\), the value of \(\mathcal{J}_X([\omega])\) depends only on the ray \(\RR_{>0}\,[\omega]\subset\mathcal{K}(X)\).

In this paper we will only consider K\"ahler classes \([\omega]\) with
\[
  c_1(X)\cdot[\omega]>0,
\]
since these are the classes that may contain PSC K\"ahler metrics. We denote the corresponding subcone by
\[
  \mathcal{K}^{+}(X):=\{[\omega]\in\mathcal{K}(X)\mid c_1(X)\cdot[\omega]>0\}.
\]
\begin{lemma}\label{lem:compacrison-J}
Let \(X\) be a compact K\"ahler manifolds with \(\mathcal K^+(X)\neq\emptyset\). Then for every \(\alpha=[\omega]\in\mathcal K^+(X)\) and every \(\omega\in\alpha\) with positive scalar curvature,
\begin{equation}\label{eq:compacrison-J}
\min_XS(\omega)\cdot\sys_2(\omega)\le\mathcal J_X(\alpha), 
\end{equation}
with equality if and only if \(\omega\) is cscK and \(\sys_2(\omega)\) is realized by a holomorphic \(1\)-cycle.
\end{lemma}

\begin{proof}
This comparison \eqref{eq:compacrison-J} follows immediately from
\[
\min_X S(\omega)\le \hat S([\omega])\quad\text{and}\quad\sys_2(\omega)\le \sysh_2([\omega]),
\]
with equality precisely when \(\omega\) is cscK and \(\sys_2(\omega)\) is realized by a holomorphic \(1\)-cycle. 
\end{proof}

Blow-ups will play an important role in what follows. Let \(\pi\colon\tilde{X}\to X\) be the blow-up of \(X\) at a point \(p\), with exceptional divisor \(E\). Given a K\"ahler class \([\omega]\) on \(X\), we consider the family of classes on \(\tilde{X}\)
\[
  [\omega_t]:=\pi^*[\omega]-t[E],\qquad t\ge0,
\]
which remains K\"ahler for all sufficiently small $t>0$. Moreover, Brown's result \cite{Brown24} shows that if $[\omega]$ contains a PSC K\"ahler metric, then so does $[\omega_t]$ for all sufficiently small $t>0$.

\begin{example}\label{ex:J1-J2}
We can compute explicitly the functional $\mathcal J_X$ on $\PP^2$ and its blow-ups at one or two points. On $\PP^2$, let $H$ denote the hyperplane class. For $\alpha=aH$, $a>0$, we have
\[
\sysh_2(\alpha)=a,\qquad c_1(\PP^2)\cdot\alpha=3a,\qquad \alpha^2=a^2,
\]
and hence $\mathcal J_{\PP^2}(\alpha)=12\pi$.

For $k=1,2$, let $X_k$ be the blow-up of $\PP^2$ at $k$ distinct points. We use the same symbol $H$ for the pullback of the hyperplane class and denote the exceptional divisors by $E_1,\ldots,E_k$. Writing
\[
\alpha=aH-\sum_{i=1}^k t_iE_i,
\]
we have
\[
\alpha^2=a^2-\sum_{i=1}^k t_i^2,\qquad c_1(X_k)\cdot\alpha=3a-\sum_{i=1}^k t_i.
\]

When $k=1$, write $E=E_1$ and $t=t_1$. The K\"ahler cone is given by $a>t>0$, and the Mori cone is generated by $E$ and
$H-E$. Thus
\[
\sysh_2(\alpha)=\min\{t,a-t\},
\]
so
\[
\mathcal J_{X_1}(\alpha)=4\pi\,\min\{t,a-t\}\frac{3a-t}{a^2-t^2}.
\]
Consequently,
\[
\max_{\alpha\in\mathcal K(X_1)}\mathcal J_{X_1}(\alpha)=\frac{20\pi}{3},
\]
with equality precisely when $\alpha$ is proportional to $H-\frac12E$.

When $k=2$, the K\"ahler cone is given by \(t_1,t_2>0,~ a>t_1+t_2\). The Mori cone is generated by the three $(-1)$-curves $E_1$, $E_2$, and $H-E_1-E_2$. Therefore,
\[
  \sysh_2(\alpha)=\min\{t_1,t_2,a-t_1-t_2\},
\]
and
\[
 \mathcal J_{X_2}(\alpha)=4\pi \min\{t_1,t_2,a-t_1-t_2\}\,\frac{3a-t_1-t_2}{a^2-t_1^2-t_2^2}.
\]
Therefore $\mathcal J_{X_2}(\alpha)\le4\pi$ whit equality holds precisely when $a=3m$ and $t_1=t_2=m$. Hence
\[
\max_{\alpha\in\mathcal K(X_2)}\mathcal J_{X_2}(\alpha)=4\pi,
\]
with equality precisely on the anticanonical ray $\RR_{>0}\,c_1(X_2)$.

It is known that $X_1$ and $X_2$ cannot admit any cscK metric due to the Matsushima--Lichnerowicz obstruction \cite{matsushima1957structure,Lichnerowicz58}. Consequently, for every PSC K\"ahler metric $\omega$ on $X_k$,
\[
\min_{X_k}S(\omega)\cdot\sys_2(\omega)<\mathcal J_{X_k}([\omega])\le
  \begin{cases}
    \displaystyle\frac{20\pi}{3}, & k=1,\\[4pt]
    4\pi, & k=2.
  \end{cases}
\]
\end{example}


\section{Proof of Theorem \ref{thm:intro-global}}\label{sec:proof1}

We first prove an algebraic estimate in arbitrary dimension which follows from Mori's bound on the length of extremal rays \cite{Mori82}.

\begin{proposition}\label{prop:RC-J}
Let $X$ be a smooth rationally connected manifold of complex dimension $n\ge1$. Then, for every $\alpha\in\mathcal K(X)$,
\begin{equation}\label{eq:RC-J-bound}
  \mathcal J_X(\alpha)\le 2\pi n(n+1).
\end{equation}
Moreover, the equality holds if and only if $X\cong\PP^n$ and \(\sysh(\alpha)\) is realized by a rational curve.
\end{proposition}

\begin{proof}
Since \(X\) is rationally connected, \(X\) is also projective. It is clear \(\sysh_2(\alpha) \neq 0\). Define
\[
D_\alpha:=\frac{n+1}{\sysh_2(\alpha)}\alpha-c_1(F).
\]
We claim that \(D_\alpha\) is nef. By Mori's cone theorem \cite[Theorem~1.4]{Mori82}, every \(K_X\)-negative extremal ray is generated by a rational curve \(C\) satisfying
\[
0<-K_X\cdot C\le r+1.
\]
For such a curve,
\[
D_\alpha\cdot C = \frac{n+1}{\sys_2^{\mathrm{hol}}(\alpha)}\alpha\cdot C-c_1(X)\cdot C\ge n+1-(-K_X\cdot C) \ge0.
\]
On the part of the Mori cone where \(K_X\ge0\), the inequality \(D_\alpha\cdot\gamma\ge0\) is immediate for all curves \(\gamma\) there because \(\alpha\) is K\"ahler and \(-c_1(X)\cdot\gamma\ge0\). Hence \(D_\alpha\) is nef.  Therefore
\[
0\le D_\alpha\cdot\alpha^{n-1}=\frac{n+1}{\sys_2^{\mathrm{hol}}(\alpha)}\alpha^n-c_1(X)\cdot\alpha^{n-1}.
\]
This proves \eqref{eq:RC-J-bound}.

If equality holds, then \(D_\alpha\cdot\alpha^{n-1}=0\).  Since \(D_\alpha\) is nef and \(\alpha\) is K\"ahler, this implies \(D_\alpha\equiv0\). Thus
\[
c_1(X)=\frac{n+1}{\sys_2^{\mathrm{hol}}(\alpha)}\alpha>0,
\]
which implies \(X\) is Fano. Moreover every rational curve \(C\subset X\) satisfies
\[
-K_X\cdot C=\frac{n+1}{\sys_2^{\mathrm{hol}}(\alpha)}\alpha\cdot C\ge n+1.
\]
Hence the pseudoindex of \(X\) is at least \(n+1\).  By the Cho--Miyaoka--Shepherd-Barron characterization \cite[Corollary~0.3]{CMSB02}, we have \(X\cong\mathbb P^n\). The converse is obvious.
\end{proof}

Now we can prove Theorem \ref{thm:intro-global}

\begin{proof}[Proof of Theorem~\ref{thm:intro-global}]
By the classification of compact K\"ahler surfaces \cite[Chapter~V]{BHPV04}, $X$ is rational or ruled, and in particular
rationally connected. By Lemma~\ref{lem:compacrison-J} and Proposition~\ref{prop:RC-J}, with $n=2$, we obtain
\[
  \min_X S(\omega)\cdot\sys_2(\omega)\le12\pi,
\]
with equality if and only \(X\cong \PP^2\), \(\omega\) is cscK and \(\sys_2(\omega)\) is realized by a rational curve. Sine \([\omega]\) is proportional to \(c_1(X)\), by \cite{BandoMabuchi87}, \(\omega\) is the Fubini--Study metric.
\end{proof}

\section{Systolic inequality on ruled surfaces}\label{sec:ruled}

We now refine the systolic inequality on ruled surfaces based on the classification. We treat the rational and non-rational ruled surfaces separately.

\subsection{Rational ruled surfaces}\label{subsec:rational-ruled}

We begin with the Hirzebruch surfaces. Let $\FF_e\to\PP^1$ be the Hirzebruch surface with invariant $e\ge0$, and denote by $C_0$ the section of self-intersection $-e$ and by $F$ the fibre class. Thus
\begin{equation} \label{eq:intersection-ruled}
      C_0^2=-e,\qquad C_0\cdot F=1,\qquad F^2=0,\qquad c_1(\FF_e)=2C_0+(e+2)F.
\end{equation}

\begin{proposition}\label{prop:ruled-J0-bound}
For every $e\ge0$,
\begin{equation}\label{eq:hirzebruch-exact-bound}
  \sup_{\alpha\in\mathcal K(\FF_e)}\mathcal J_{\FF_e}(\alpha)=4\pi\,\frac{e+4}{e+2}.
\end{equation}
The maximum is attained precisely on the class proportional to $C_0+(e+1)F$.
\end{proposition}

\begin{proof}
Write $\alpha=aC_0+bF$. The K\"ahler cone is given by $a>0$ and $b>ae$ thanks to Kleiman's criterion. Since the effective cone is generated by $C_0$ and $F$, which form an integral basis, we have
\[
  \sysh_2(\alpha)=\min\{a,b-ae\}.
\]
Moreover, we compute that
\[
  \alpha^2=2ab-ea^2, \qquad c_1(\FF_e)\cdot\alpha=2b+(2-e)a.
\]
Writing $x:=(b-ae)/a>0$, we obtain
\begin{equation}\label{eq:hirzebruch-one-variable}
  \frac{\mathcal J_{\FF_e}(\alpha)}{4\pi}=\min\{1,x\}\,\frac{e+2x+2}{e+2x}.
\end{equation}
For $0<x\le1$, the right-hand side equals $x+2x/(e+2x)$ and is strictly increasing. For $x\ge1$, it equals $1+2/(e+2x)$ and is strictly decreasing. The maximum is therefore attained exactly when $x=1$, or equivalently $b=(e+1)a$, with value $(e+4)/(e+2)$.
\end{proof}

We next deal with the blowups of the Hirzebruch surfaces. Let \(\rho(X)\) denote the Picard number of \(X\). It is known that \(\rho(\FF_e)=2\) and \(\rho(\bl_k \FF_e)=2+k\) for all \(e\ge 0\) where \(\bl_k \FF_e\) is the blowup of \(\FF_e\) at \(k\) points since each point blow-up increases the Picard number by one.

\begin{lemma}\label{lem:rational-nef-class}
Let $X$ be a smooth rational surface with $\rho(X)\ge3$, and let $\alpha\in\mathcal K^+(X)$. Then \(\alpha-\sysh_2(\alpha)\cdot c_1(X)\) is nef. In particular,
\begin{equation}\label{eq:rational-K-square}
\mathcal J_X(\alpha)\le2\pi\left(1+\frac{m^2K_X^2}{\alpha^2}\right)\le 4\pi.
\end{equation}
\end{lemma}
\begin{proof}
Write \(\beta:=\alpha-\sysh_2(\alpha)\cdot c_1(X)\). Since $\rho(X)\ge3$, every contraction of $K_X$-negative extremal rays is birational. Hence, each $K_X$-negative extremal ray is generated by a $(-1)$-curve $E_i$. Since \(c_1(X)\cdot E_i=1\) we have
\[
\beta\cdot E_i=\alpha\cdot E_i-\sysh_2(\alpha)\ge0.
\]
For $\gamma\in\mori(X)_{K_X\ge0}$ we also have
\[
\beta\cdot\gamma=\alpha\cdot\gamma+\sysh_2(\alpha)\cdot K_X\cdot\gamma\ge0.
\]
By Mori's cone theorem \cite{Mori82}, we have $\beta$ is nef. Consequently,
\[
0\le\beta^2=\alpha^2-2\sysh_2(\alpha)\cdot c_1(X)\cdot\alpha+m^2K_X^2,
\]
which proves \eqref{eq:rational-K-square} thanks to the Hodge index theorem.
\end{proof}

\begin{theorem}\label{thm:hirzebruch}
Let $X\to\PP^1$ be a rational ruled surface, and let $\omega$ be a PSC K\"ahler metric on $X$. Then
\begin{equation} \label{eq:rational-bound}
    \min_X S(\omega)\cdot\sys_2(\omega)\le8\pi.
\end{equation}
Moreover, equality holds if and only if $X\cong\PP^1\times\PP^1$ with the product of Fubini--Study metrics up to scaling.
\end{theorem}

\begin{proof}
Write $\alpha=[\omega]$. If $\rho(X)=2$, then $X\cong\FF_e$ for some $e\ge0$. By Proposition~\ref{prop:ruled-J0-bound}, we have
\[
\mathcal J_X(\alpha)\le4\pi\,\frac{e+4}{e+2}\le8\pi.
\]
If $\rho(X)\ge3$, by Lemma~\ref{lem:rational-nef-class} we have $\mathcal J_X(\alpha)\le4\pi$. This proves the inequality thanks to Lemma~\ref{lem:compacrison-J}.

If equality holds, then necessarily $X\cong\FF_0$ and $\alpha=a(C_0+F)$ for some $a>0$. Equality also forces $\omega$ to be cscK. Since \(\alpha\) is proportional to $c_1(X)=2(C_0+F)$, again, by \cite{BandoMabuchi87}, \(\omega\) is the product Fubini--Study metric, up to scaling. The converse is obvious.
\end{proof}

\subsection{Non-rational ruled surfaces}\label{subsec:nonrational-ruled}

Let $X_0=\PP(\mathcal E)\to B$ be a geometrically ruled surface over a smooth curve of genus $g\ge1$. Choose a section $C_0$ of minimal self-intersection and write $C_0^2=-e$ where \(e \in \mathbb Z\) is the invariant. Let $F$ denoting the fibre class, we have
\[
  C_0\cdot F=1,\qquad F^2=0,\qquad c_1(X_0)=2C_0+(2-2g+e)F.
\]

Let $X_k$ be obtained from $X_0$ by a sequence of $k$ point blow-ups. We allow infinitely near centres and use $E_i$ for the total transform on $X_k$ of the exceptional divisor of the $i$-th blow-up. We use the same symbols $C_0,F$ for their pullbacks to $X_k$. Thus
\[
  E_i\cdot E_j=-\delta_{ij},\qquad C_0\cdot E_i=F\cdot E_i=0.
\]
Write a K\"ahler class as
\begin{equation} \label{eq:Kahler-nonrational}
      \alpha=aC_0+bF-\sum_{i=1}^k t_iE_i,\qquad a=\alpha\cdot F>0.
\end{equation}
The relevant intersection numbers are
\begin{equation}\label{eq:nonrational-intersections}
\alpha^2=2ab-ea^2-\sum_{i=1}^k t_i^2,\qquad c_1(X_k)\cdot\alpha =2b+(2-2g-e)a-\sum_{i=1}^k t_i.
\end{equation}

\begin{lemma}\label{lem:Jk-ruled-g>=1}
Let \(X_k\) be a non-rational ruled surfaces and let $\alpha\in\mathcal K^+(X_k)$. Then
\begin{equation} \label{eq:nonrational-bound}
    \mathcal J_{X_k}(\alpha)\le4\pi.
\end{equation}
Moreover, 
\begin{enumerate}
    \item If $k=0$, equality holds precisely when $g=1$ and $\sysh_2(\alpha)$ is realized by a \(\PP^1\)-fibre.
    \item If $k\ge1$, then we have
\begin{equation}\label{eq:nonrational-blowup-bound}
  \mathcal J_{X_k}(\alpha)<2\pi.
\end{equation}
\end{enumerate}
\end{lemma}

\begin{proof}
Write \(\alpha\in\mathcal K^+(X_k)\) as in \eqref{eq:Kahler-nonrational}. By \eqref{eq:nonrational-intersections}, we have
\begin{equation}\label{eq:ruled-key-identity}
  a\,c_1(X_k)\cdot\alpha-\alpha^2=(2-2g)a^2-\sum_{i=1}^k t_i(a-t_i).
\end{equation}

Assume first $k=0$. Since $g\ge1$, we have
\[
\mathcal J_{X_0}(\alpha)=4\pi \sysh_2(\alpha)\,\frac{c_1(X_0)\cdot\alpha}{\alpha^2}\le4\pi\frac{\sysh_2(\alpha)}{a}\le4\pi.
\]
where the last inequality follows from \(\sysh_2(\alpha) \le \alpha \cdot F =a\). Moreover, equality holds exactly when $g=1$ and $\sysh_2(\alpha)=a$.

Assume now $k\ge1$. Since an exceptional divisor \(E_i\) and the strict transform \(F_i-E_i\) for a fibre \(F_i\) contains a blowup point are candidates for \(\sysh_2(\alpha)\), we have
\begin{equation}\label{eq:fibre-exceptional-areas}
0<t_i<a,\qquad \sysh_2(\alpha)\le\min\{t_i,a-t_i\}\le\frac a2.
\end{equation}
Therefore, the right-hand side of \eqref{eq:ruled-key-identity} is strictly negative, and \eqref{eq:fibre-exceptional-areas} yields
\[
\mathcal J_{X_k}(\alpha)<4\pi\frac{\sysh_2(\alpha)}{a}\le2\pi.
\]
This completes the proof.
\end{proof}

\begin{theorem}\label{thm:ruled-g>=1}
Let $X\to B$ be a non-rational ruled surface based over a smooth curve of genus $g\ge1$, and let $\omega$ be a PSC K\"ahler metric on $X$. Then
\begin{equation} \label{eq:nonrational}
  \min_X S(\omega)\cdot\sys_2(\omega)\le4\pi.   
\end{equation}
Moreover, equality holds if and only if $B$ is elliptic, $X$ is minimal, $(X,\omega)$ is holomorphically and isometrically covered by $\PP^1\times\CC$ with a product of a Fubini--Study metric and a flat metric, and the $2$-systole is realized by a fibre.
\end{theorem}

\begin{proof}
Write $\alpha=[\omega]$ and let $a=\alpha\cdot F$. Combining Lemma~\ref{lem:compacrison-J} with Lemma~\ref{lem:Jk-ruled-g>=1}, we have
\[
  \min_X S(\omega)\cdot\sys_2(\omega)\le\mathcal J_X(\alpha)\le4\pi.
\]
If equality holds, Lemma~\ref{lem:Jk-ruled-g>=1} shows that $X$ is minimal, $g=1$, and $\sysh_2(\alpha)=a$. Moreover, equality in Lemma~\ref{lem:compacrison-J} then gives
\[
  S(\omega)\equiv\hat S(\alpha),\qquad \sys_2(\omega)=\sysh_2(\alpha)=a.
\]
In particular, the $2$-systole is realized by a fibre.

Write $X=\PP(\mathcal E)$ over the elliptic curve $B$. By the classification of cscK metrics on minimal ruled surfaces \cite[Theorem~2 and Remark~1(1)]{ApostolovTonnesen06}, $\mathcal E$ is polystable and every cscK metric on $X$ is, up to a biholomorphism, locally symmetric. When the base curve \(B\) is elliptic, its universal cover is $\PP^1\times\CC$ with the product of a Fubini--Study metric and a flat metric. This gives the stated metric characterization.

Conversely, under the stated conditions $\omega$ is cscK and $\sys_2(\omega)=a$. Since $g=1$ and $X$ is minimal,
\eqref{eq:ruled-key-identity} gives $a\,c_1(X)\cdot\alpha=\alpha^2$. Therefore
\[
  S(\omega)=\hat S(\alpha)=\frac{4\pi}{a}, \qquad \min_X S(\omega)\cdot\sys_2(\omega)=4\pi,
\]
which completes the proof.
\end{proof}

\section{Level set method on non-rational PSC K\"ahler surfaces}\label{sec:level-set}

In \cite{stern2022scalar}, Stern introduced the following inequality for a non-constant $S^1$-valued harmonic map $u$ on a $3$-manifold $(M,g)$ through the level set method:
\begin{equation}\label{Stern-ineq}
2 \pi \int_{\theta \in S^1} \chi\left(\Sigma_\theta\right)d\theta
\;\geq\; \frac{1}{2} \int_{\theta \in S^1} \left[\int_{\Sigma_\theta}
\left(|d u|^{-2}|\operatorname{Hess}(u)|^2+\scal_M(g)\right)dV_g \right]d\theta,
\end{equation}
and used it to give a new proof of the Bray--Brendle--Neves systolic inequality for the $2$-systole. In this section, we adapt the level set method to non-rational PSC Kähler surfaces, and obtain an alternative proof of Theorem~\ref{thm:ruled-g>=1}.

\subsection{A fibre-wise scalar curvature inequality}
Let $X$ be a compact non-rational PSC Kähler surface. By the classification result recalled in the introduction, $X$ is a ruled surface fibred over a compact Riemann surface $B$ of genus $g(B)\ge1$. Denote by
\[
f\colon (X,\omega)\longrightarrow B
\]
the induced non-constant holomorphic fibration. By the uniformization theorem, we may equip $B$ with a constant curvature metric $\omega_0$ of non-positive Gaussian curvature (so $\omega_0$ is hyperbolic if $g(B)\ge2$ and flat if $g(B)=1$).

Fix a point $z\in B$. Then $f^{-1}(z)$ is a (possibly singular) Cartier divisor on $X$, which we denote by $D_z$. It defines a line bundle $\mathcal{O}(D_z)$ whose first Chern class is represented by $f^*\omega_0$ (after a suitable normalization). In what follows we restrict attention to regular values of $f$, so that $D_z$ is smooth, and we keep the notation $D_z$ for the smooth fibre. Recall that for a smooth divisor $D\subset X$, the adjunction formula states
\[
K_D = \bigl(K_X\otimes\mathcal{O}(D)\bigr)|_D.
\]
Since \(D=D_z\) is a fibre of the holomorphic fibration \(f\colon X\to B\), its normal bundle \(\mathcal N_D\cong \mathcal O(D)|_D\) is holomorphically trivial. By the adjunction formula,
\[
K_D=\bigl(K_X\otimes \mathcal O(D)\bigr)|_D.
\]
Taking first Chern classes, we obtain
\[
c_1(D)=c_1(X)|_D.
\]
Therefore, for the Ricci forms induced by \(\omega\), the difference
\[
\ric_D(\omega)-\ric_X(\omega)|_D
\]
is an exact real \((1,1)\)-form on the compact Riemann surface \(D\). By the \(\partial\bar\partial\)-lemma, there exists a smooth real-valued function \(\phi\in C^\infty(D)\) such that
\[
\ric_D(\omega)=\ric_X(\omega)|_D+\sqrt{-1}\,\partial_D\bar\partial_D\phi.
\]
Tracing with respect to the induced metric \(\omega|_D\), we get
\[
S_D(\omega)
=
\operatorname{tr}_{\omega|_D}\bigl(\ric_X(\omega)|_D\bigr)
+\Delta_D\phi.
\]
Let \(\nu\) be a local unit normal vector field of type \((1,0)\) along \(D\) with respect to \(\omega\). Since \(D\subset X\) is a complex hypersurface, we have
\[
\operatorname{tr}_{\omega|_D}\bigl(\ric_X(\omega)|_D\bigr)
=
S_X(\omega)-\ric_X(\omega)(\nu,\bar\nu).
\]
Hence
\[
S_D(\omega)
=
S_X(\omega)-\ric_X(\omega)(\nu,\bar\nu)+\Delta_D\phi.
\]
In particular, since
\[
\nu=\frac{\nabla^{1,0}f}{|\nabla^{1,0}f|},
\]
we obtain
\begin{equation}\label{traced_Gauss}
\ric_X(\omega)\bigl(\nabla^{1,0}f,\nabla^{0,1}f\bigr)
=
|\nabla^{1,0}f|^2\bigl(S_X(\omega)-S_D(\omega)+\Delta_D\phi\bigr).
\end{equation}

We next recall the Bochner formula for holomorphic maps and a co-area formula adapted to the present setting.

\begin{lemma}[Bochner formula]
Let \(f \colon (X,\omega) \rightarrow (N,\tilde{\omega})\) be a holomorphic map between Kähler manifolds. Then
\[
\Delta |\partial f|^2
= |\nabla \partial f|^2
  + \big\langle \ric(\omega), f^* \tilde{\omega} \big\rangle
  - \operatorname{tr}^2_{\omega} \bigl(f^*\operatorname{Rm}(\tilde{\omega})\bigr),
\]
where \(\ric(\omega)\) is the Ricci form of \(X\) and \(\operatorname{Rm}(\tilde{\omega})\) is the curvature form of \(N\).
\end{lemma}

\begin{lemma}[Co-area formula]
Let \((X^n,\omega)\) be a compact Kähler manifold and let \((B,\omega_0)\) be a compact Riemann surface. Let \(f\colon X\to B\) be a non-constant holomorphic map, and let \(g\in C^\infty(X)\).
Then
\[
\int_X g\,\omega^n= n\int_B \left(\int_{f^{-1}(z)}\frac{g}{|\partial f|^2}\,\omega^{n-1}\right)\omega_0.
\]
\end{lemma}

When \((N,\tilde{\omega})=(B,\omega_0)\), the curvature term \(\operatorname{tr}^2_{\omega}(f^*\operatorname{Rm}(\omega_0))\) is non-positive and vanishes identically if and only if \((B,\omega_0)\) is flat. In particular,
\[
\Delta |\partial f|^2\;\ge\; |\nabla \partial f|^2+ \ric_X(\omega)\bigl(\nabla^{1,0}f,\nabla^{0,1}f\bigr),
\]
with equality if and only if \((B,\omega_0)\) is an elliptic curve with a flat metric. Combining this with \eqref{traced_Gauss} we obtain
\begin{equation}\label{important}
\Delta |\partial f|^2\;\ge\; |\nabla \partial f|^2+ |\nabla^{1,0}f|^2\bigl(S_X(\omega)-S_D(\omega)+\Delta_D\phi\bigr).
\end{equation}

We now integrate this inequality in the fibre direction using the co-area formula. The following lemma is the basic Stern-type inequality that we shall use in the non-rational ruled case.

\begin{lemma}\label{lem:stern-type}
Let \((X^{n},\omega)\) be a compact Kähler manifold and let \((B,\omega_0)\) be a compact Riemann surface of genus \(g(B)\ge1\), endowed with a constant curvature metric \(\omega_0\). Suppose that \(f\colon X \rightarrow B\) is a non-constant holomorphic map, and let \(D=f^{-1}(z)\) denote a regular fibre. Then
\begin{align}\label{eq:levelset_ineq}
\int_{B} \left[\int_{D} \left(\frac{\left|\nabla \partial f\right|^2}{|\partial f|^2}+  S_X(\omega) -  S_{D}(\omega) \right)\omega^{n-1} \right]\omega_0 \;\le\; 0.
\end{align}
Moreover, equality holds in \eqref{eq:levelset_ineq} if and only if \(g(B)=1\) and \((B,\omega_0)\) is an elliptic curve with a flat metric.
\end{lemma}

\begin{proof}
Write \(B=A \cup B_0\), where \(A\) is an open neighbourhood of the finite set of critical values of \(f\), and \(B_0\) consists only of regular values. Integrating \eqref{important} over \(f^{-1}(B_0)\), we obtain
\begin{align}\label{eq:stern-ineq-local}
\int_{f^{-1}(B_0)}\Bigl(\bigl|\nabla \partial f\bigr|^2+ \bigl|\nabla^{1,0}f\bigr|^2\bigl(S_X(\omega) - S_{D}(\omega)+\Delta_D\phi\bigr)\Bigr)\omega^n\;\le\; \int_{f^{-1}(B_0)} \bigl(\Delta |\partial f|^2\bigr) \omega^n.
\end{align}
Since \(X\) is compact and has no boundary,
\[
\int_X \bigl(\Delta |\partial f|^2\bigr)\omega^n=0.
\]
Hence, by choosing \(A\) with arbitrarily small measure and using Sard's theorem, we may pass to the limit in \eqref{eq:stern-ineq-local} and obtain
\[
\int_X\Bigl(\bigl|\nabla \partial f\bigr|^2+ \bigl|\nabla^{1,0}f\bigr|^2\bigl(S_X(\omega) - S_{D}(\omega)+\Delta_D\phi\bigr)\Bigr)\omega^n\;\le\; 0.
\]
Now apply the co-area formula term by term. Since \(|\nabla^{1,0}f|^2=|\partial f|^2\), we get
\begin{align*}
\int_X |\nabla\partial f|^2\,\omega^n
&=
n\int_B\left(\int_D \frac{|\nabla\partial f|^2}{|\partial f|^2}\,\omega^{n-1}\right)\omega_0,\\
\int_X |\nabla^{1,0}f|^2\bigl(S_X(\omega)-S_D(\omega)\bigr)\,\omega^n
&=
n\int_B\left(\int_D \bigl(S_X(\omega)-S_D(\omega)\bigr)\,\omega^{n-1}\right)\omega_0,
\end{align*}
and similarly
\[
\int_X |\nabla^{1,0}f|^2\,\Delta_D\phi\,\omega^n
=
n\int_B\left(\int_D \Delta_D\phi\,\omega^{n-1}\right)\omega_0.
\]
Since each regular fibre \(D\) is compact without boundary,
\[
\int_D \Delta_D\phi\,\omega^{n-1}=0.
\]
Therefore
\begin{align*}
0
&\ge
n\int_B\left[\int_D\left(
\frac{|\nabla\partial f|^2}{|\partial f|^2}
+S_X(\omega)-S_D(\omega)\right)\omega^{n-1}\right]\omega_0,
\end{align*}
which gives \eqref{eq:levelset_ineq}. The equality statement follows from the discussion before \eqref{important}: equality in the Bochner inequality holds if and only if \(\operatorname{Rm}(\omega_0)\equiv0\), namely if and only if \(g(B)=1\) and \((B,\omega_0)\) is flat.
\end{proof}

\subsection{The 2-systole on non-rational PSC K\"ahler surfaces}

In this subsection, we study the homological $2$-systole on non-rational PSC K\"ahler surfaces. Recall that, by the classification of PSC K\"ahler surfaces, a non-rational PSC K\"ahler surface is precisely a (possibly blown-up) ruled surface fibred over a curve of genus $g\ge1$.

By leveraging \eqref{eq:levelset_ineq}, we provide an alternative proof of Theorem \ref{thm:ruled-g>=1} with an analytic method.
It is worth noting that for a K\"ahler metric the Chern scalar curvature $S(\omega)$ differs from the Riemannian scalar curvature $\scal(g_\omega)$ by a factor $2$.
\begin{theorem}\label{thm:ruled-g>=1-levelset}
Let $(X,\omega)$ be a non-rational PSC K\"ahler surface admitting a holomorphic fibration $ X\to B$ to a compact Riemann surface $B$ with genus $g(B)\ge1$. Then
\begin{equation}\label{main_ineq}
\min_X S_X(\omega)\cdot\sys_2(\omega) \;\leq\; 4\pi.
\end{equation}
Moreover, equality holds if and only if $g(B)=1$, $X$ is covered by $\PP^1\times \CC$ equipped with the product of the Fubini--Study metric on $\PP^1$ and a flat metric on $\CC$, in such a way that $\sys_2(\omega)$ is achieved by the $\PP^1$-fibre.
\end{theorem}

\begin{proof}
Let $f\colon X\to B$ be the holomorphic fibration, and let $\omega_0$ be a constant curvature metric on $B$ of non-positive Gaussian curvature, so that \eqref{eq:levelset_ineq} holds:
    \begin{align*}
        \int_{B} \left[\int_{D} \left(
        \frac{\left|\nabla \partial f\right|^2}{|\partial f|^2}
        +S_X(\omega)\right)\omega \right]\omega_0
        \le  \int_{B} \left(\int_{D} S_{D}(\omega)\,\omega \right)\omega_0.
    \end{align*}
For every regular value $z$, the fibre $D$ is a smooth rational curve, hence $D\cong\PP^1$ and $\chi(D)=2$. By Gauss--Bonnet formula,
    \[
        \int_{D} S_{D}(\omega)\,\omega
        = 2\pi\,\chi(D) = 4\pi.
    \]
    Integrating over $B$, we obtain
    \begin{align*}
       4\pi\int_{B}\omega_0 =2\pi\int_{B} \chi(D)\,\omega_0
       &= \int_{B} \left(\int_{D} S_{D}(\omega)\,\omega \right)\omega_0  \\     
       &\geq \int_{B} \left(\int_{D} S_{X}(\omega)\,\omega \right)\omega_0 \\
       &\geq \min_X S_X(\omega)\cdot\int_{B} \Vol_\omega(D)\,\omega_0.
    \end{align*}
    By definition of the homological $2$-systole we have
    \[
      \Vol_\omega(D)\;\geq\;\sys_2(\omega)
    \]
    for every regular fibre $D$. Hence
    \begin{align*}
        4 \pi \int_{B} \omega_0
        &\geq \min_X S_X(\omega)\cdot\int_{B} \Vol_\omega(D)\,\omega_0 \\
        &\geq \min_X S_X(\omega)\cdot\sys_2(\omega)\int_{B}  \omega_0 ,
    \end{align*}
    and since $\int_B\omega_0>0$ this yields
    \[
      \min_X S_X(\omega)\cdot\sys_2(\omega)\;\le\;4\pi,
    \]
as claimed. The equality case happens if $B$ is an elliptic curve endowed with a flat metric, $\nabla f$ is parallel, which implies $X$ is isometrically covered by $\PP^1 \times \CC$, and $(X,\omega)$ is cscK so that $\sys_2(\omega)$ is realized by $\PP^1$-fibre. 
\end{proof}

\begin{corollary}
Let $(X,\omega)$ be a non-rational PSC K\"ahler surface admitting a non-constant holomorphic map $f\colon X\to B$ to a compact Riemann surface $B$ with genus $g(B)\ge2$. Then
\[
\min_X S_X(\omega)\cdot\sys_2(\omega) \;<\; 4\pi.
\]
\end{corollary}

\begin{proof}
If $g(B)\ge2$, then $B$ is hyperbolic and cannot carry a flat metric. Hence equality in \eqref{main_ineq} cannot occur in Lemma~\ref{lem:stern-type}, nor in Theorem~\ref{thm:ruled-g>=1-levelset}, and the inequality in Theorem~\ref{thm:ruled-g>=1-levelset} is strict.
\end{proof}

\begin{example}
Let $X=\PP^1\times B$ be a compact complex surface, where $B$ is a compact Riemann surface of genus $g(B)\ge2$. Equip $X$ with the product K\"ahler metric
\[
\omega=\omega_{\mathrm{FS}}\oplus\omega_B,
\]
where on $\PP^1$ we take the Fubini--Study metric normalized by
\[
\Vol_{\omega_{\mathrm{FS}}}(\PP^1)=\pi,\qquad S_{\PP^1}(\omega_{\mathrm{FS}})\equiv4,
\]
and on $B$ we choose a constant Chern scalar curvature metric with
\[
S_B(\omega_B)=-4+\varepsilon\quad\text{for some }\varepsilon\in(0,4).
\]
Then the Chern scalar curvature of the product metric is constant and given by
\[
S_X(\omega)=S_{\PP^1}(\omega_{\mathrm{FS}})+S_B(\omega_B)=4+(-4+\varepsilon)=\varepsilon,
\]
so $\min_X S_X(\omega)=\varepsilon>0$ and $X$ has positive scalar curvature in our convention.

Next we compare the areas of the two basic complex curves:
\begin{itemize}
  \item For the $\PP^1$-fibre $F=\PP^1\times\{p\}$, calibration by $\omega$
  gives
  \[
  \Vol_\omega(F)=\Vol_{\omega_{\mathrm{FS}}}(\PP^1)=\pi.
  \]
  \item For the $B$-fibre $B_p=\{q\}\times B$, Gauss--Bonnet for the Chern scalar curvature gives
  \[
  \int_B S_B(\omega_B)\,\omega_B=2\pi\chi(B)=2\pi(2-2g(B))=4\pi(1-g(B)).
  \]
  Since $S_B(\omega_B)\equiv-4+\varepsilon<0$, we obtain
  \[
  \Vol_\omega(B_p)=\Vol_{\omega_B}(B)=\frac{4\pi(g(B)-1)}{4-\varepsilon}.
  \]
\end{itemize}
For $g(B)\ge2$ and $\varepsilon\in(0,4)$ one has \(\Vol_\omega(B_p)>\pi\), so the $2$-systole is realized by the $\PP^1$-fibre:
\[
\sys_2(\omega)=\min\bigl\{\Vol_\omega(F),\Vol_\omega(B_p)\bigr\}=\pi.
\]
Consequently,
\[
\min_X S_X(\omega)\cdot\sys_2(\omega)=\varepsilon\cdot\pi \;<\; 4\pi.
\]
In particular, this product is independent of the genus $g(B)$, and it can be made arbitrarily close to $4\pi$ by letting $\varepsilon\nearrow 4$.
\end{example}

\bibliographystyle{amsalpha}
\bibliography{wpref}

\bigskip
  \footnotesize

  Zehao Sha, \textsc{C. N. Yang Institute for Mathematics and Fundamental Physics, Shanghai, China}\par\nopagebreak
  Email address: \texttt{zhsha@imfp.org.cn}\par\nopagebreak
  Homepage: \url{https://ricciflow19.github.io/}

\end{document}